\documentclass[12pt]{amsart}

\usepackage{amsfonts,amsmath,amsthm}
\usepackage{latexsym}

\newtheorem{theorem}{Theorem}[section]
\newtheorem{corollary}[theorem]{Corollary}
\newtheorem{lemma}[theorem]{Lemma}

\newtheorem{conjecture}[theorem]{Conjecture}
\newtheorem{proposition}[theorem]{Proposition}

\theoremstyle{definition}

\newcommand{\ZZ}{\ensuremath{\mathbb{Z}}}
\newcommand{\RR}{\ensuremath{\mathbb{R}}}

\newcommand{\cg}{\ensuremath{\mathcal{G}}}

\newcommand{\vb}{\ensuremath{\mathbf{b}}}

\newcommand{\vu}{\ensuremath{\mathbf{u}}}

\newcommand{\vh}{\ensuremath{\mathbf{h}}}
\newcommand{\va}{\ensuremath{\mathbf{a}}}
\newcommand{\vx}{\ensuremath{\mathbf{x}}}

\def \< {\langle}
\def \> {\rangle}

\begin{document}

\title[A Walsh-Fourier approach to circulant Hadamards]{A Walsh-Fourier approach to the circulant Hadamard conjecture}

\author[M. Matolcsi]{M. Matolcsi}
\address{M. M.: Alfr\'ed R\'enyi Institute of Mathematics,
Hungarian Academy of Sciences POB 127 H-1364 Budapest, Hungary
Tel: (+361) 483-8307, Fax: (+361) 483-8333}
\email{matolcsi.mate@renyi.mta.hu}

\thanks{The author was supported bt OTKA grant No.\! 109789 and by ERC-AdG 321104.}

\begin{abstract}
We describe an approach to the circulant Hadamard conjecture based on Walsh-Fourier analysis. We show that the existence of a circulant Hadamard matrix of order $n$ is equivalent to the existence of a non-trivial solution of a certain homogenous linear system of equations. Based on this system, a possible way of proving the conjecture is proposed.
\end{abstract}

\maketitle

\bigskip

\section{introduction}

A {\it real Hadamard matrix} is a square matrix with $\pm 1$ entries such that the rows (and thus columns) are pairwise orthogonal. A {\it circulant (or cyclic) matrix} $C$ is a square matrix which is generated by the cyclic permutations of a row vector, i.e. there exists a vector $\vx=(x_1, \dots x_{n})$ such that $c_{i,j}=x_{j-i+1}$ for $1\le i, j\le n$ (the difference being reduced mod $n$ to the set $\{1, \dots , n\}$; we prefer to use the indices $1, \dots, n$ rather than $0, \dots, n-1$).

\medskip

It is trivial to check that the $4\times 4$ circulant matrix generated by the row vector $(-1, 1,1,1)$ is Hadamard. However, no circulant Hadamard matrix of order larger than 4 is known. The following famous conjecture was made by Ryser \cite{ryser}, more than 50 years ago :

\begin{conjecture}{\rm{(Circulant Hadamard conjecture)}}
For $n>4$ there exists no $n\times n$ circulant real Hadamard matrix.
\end{conjecture}

The first significant result concerning this conjecture was made by R. J. Turyn \cite{turyn} using arguments from algebraic number theory. He proved that if a circulant Hadamard matrix of order $n$ exists then $n$ must be of the form $n=4u^2$ for some odd integer $u$ which is not a prime-power. The most powerful breakthroughs were later obtained by the "field descent method" of B. Schmidt \cite{schmidt1, schmidt2} and its extensions by K. H. Leung and B. Schmidt \cite{ls1, ls2}. Currently, the smallest open case is $n=4u^2$ with $u=11715$, and there are less than 1000 remaining open cases in range $u\le 10^{13}$.

\medskip

In this note we offer a more elementary approach to the circulant Hadamard conjecture, based on Walsh-Fourier analysis.

\section{A Walsh-Fourier approach}

The approach described in this note is inspired by the results of \cite{mubmols}, where a Fourier analytic approach to the problem of mutually unbiased bases (MUBs) was presented. The basic idea is that the Fourier transform is capable of turning non-linear conditions into linear ones.

\medskip

We briefly introduce the necessary notions and notations here. Let $\ZZ_2$ denote the cyclic group of order $2$, and let $\cg = \ZZ_2^n$. An element of $\cg$ will be regarded as a column vector of length $n$ whose entries are $\pm 1$. And vice versa, each such column vector will be regarded as an element of $\cg$. Accordingly, an $n\times n$ matrix $A$ containing $\pm 1$ entries will be regarded as an $n$-element subset of $\cg$, the columns of $A$ being the elements. We will use (Walsh)-Fourier analysis on $\cg$. Let $\hat{\cg}$ denote the dual group. Then $\hat{\cg}$ is isomorphic to $\ZZ_2^n$ and an  element $\gamma$ of $\hat{\cg}$ will be identified with a row vector containing 0-1 entries. The action of a character $\gamma = (\gamma_1, \dots \gamma_{n} )\in \hat{\cg}$ on an element $\vx=(x_1, \dots  x_{n})\in \cg$ is defined as $\gamma(\vx)=\vx^\gamma=x_1^{\gamma_1} \dots x_n^{\gamma_n}$. We will also use the notation $\hat{\cg}_0$ for the subgroup of elements $\gamma\in \hat{\cg}$ such that $\gamma_1 + \gamma_2 +\dots + \gamma_n \equiv 0$ (mod $2$).

\medskip

Let $A$ be any  $n\times n$ matrix containing $\pm 1$ entries, and let $\va_1, \dots ,\va_n$ denote the columns of $A$. The Fourier transform of (the indicator function of) $A$ will be defined as $\hat{A} (\gamma)=\sum_{j=1}^n \gamma (\va_j)=\sum_{j=1}^n \va_j^\gamma$. This is our main object of study. Notice here that

\begin{equation}\label{eq1}
|\hat{A} (\gamma)|^2=\sum_{j,k=1}^n (\va_j/\va_k)^\gamma ,
\end{equation}
where the quotient $\va_j/\va_k$ is understood coordinate-wise, i.e. $\va/\vb =(a_1/b_1, \dots , a_n/b_n)$. (As long as we work with $\pm 1$ entries the operation division can be replaced by multiplication, but we prefer to use division in the notation because it can also be used in the more general context of complex Hadamard matrices.)

\medskip

To illustrate the use of the Fourier transform $\hat{A} (\gamma)$, let me include here a neat proof of the fact that an $n\times n$ Hadamard matrix can only exist if 4 divides $n$. There is an easy combinatorial proof of this fact, but i believe that the Fourier proof is the "book proof".

\begin{proposition}\label{had}
If an $n\times n$ real Hadamard matrix exists, then 4 divides n, or $n=1,2$.
\end{proposition}
\begin{proof}
Let $H$ be an $n\times n$ real Hadamard matrix. If $n>1$ then $n$ must clearly be even. Assume $2|n$, but $n$ is not divisible by 4.

\medskip

As described above, the columns $\vh_1, \dots \vh_n$ of $H$ can be regarded as elements of $\cg=\ZZ_2^n$ and for any $0 - 1$ vector $\gamma\in \hat{\cg}$ we have $\hat{H} (\gamma)=\sum_{j=1}^n \vh_j^\gamma$ ,and
\begin{equation}\label{hhat2}
|\hat{H} (\gamma)|^2=\sum_{j, k=1}^n (\vh_j/\vh_k)^\gamma.
\end{equation}

Clearly, $|\hat{H} (\gamma)|^2\ge 0$ for all $\gamma$. However, consider the element $\gamma=(1,1,\dots ,1)$. On the right hand side of \eqref{hhat2} we have $1$ if $j=k$, and $-1$ if $j\ne k$ (here we use the fact that 4 does not divide $n$). Therefore, the right hand side evaluates to $n-n(n-1)=-n(n-2)$, which is negative if $n>2$, a contradiction.
\end{proof}

\medskip

Let us now turn to circulant Hadamard matrices. Assume $\vu=(u_1, \dots u_n)$ is a $\pm 1$ vector which which generates a circulant Hadamard matrix $H$. Consider the function
\begin{equation}\label{Mgamma}
M(\gamma)=\vu^\gamma
\end{equation}
where $\gamma$ ranges over $\hat\cg=\ZZ_2^n$. Let $\pi_j\in \hat{\cg}$ denote the element with an entry 1 at coordinate $j$, and all other entries being 0.

\medskip

We have the following properties of the function $M$:

\begin{equation}\label{Mpm1}
M(\gamma)=\pm 1   \ \textrm{for \ all} \ \gamma \in \ZZ_2^n, \ \textrm{and} \ M(0)=1.
\end{equation}
This is trivial.

\medskip

For all $d=1, \dots n/2$, and all $\gamma \in \ZZ_2^n$ we have
\begin{equation}\label{Mtiling}
\sum_{j-k=d (mod \ n)} M(\gamma+\pi_j+\pi_k)=0.
\end{equation}
This is a consequence of the cyclic orthogonality property: $\sum_{j=1}^n u_j u_{j+d}=0$. Spelling it out:
\begin{equation*}
\sum_{j-k=d (mod \ n)} M(\gamma+\pi_j+\pi_k)=\sum_{j=1}^n \vu^{\gamma+\pi_j+\pi_{j+d}}=\vu^\gamma \sum_{j=1}^n u_ju_{j+d}=0.
\end{equation*}

\medskip

The aim is to get a contradiction from the facts \eqref{Mpm1}, \eqref{Mtiling} for $n>4$. If we just consider the conditions \eqref{Mtiling}, and regard each $M(\gamma)$ as a {\it real variable} then we have a homogenous system of linear equations with $2^n$ variables and $\frac{n}{2}2^{n}$ linear constraints. We will prove that this is an equivalent formulation of the circulant Hadamard conjecture, i.e. the existence of any non-trivial solution to this linear system of equations implies the existence of a circulant Hadamard matrix of order $n$. We will first need some intermediate lemmas.

\begin{lemma}\label{trivS}
The circulant Hadamard conjecture is true for $n$ if and only if the $n$-variable equation
\begin{equation}\label{S}\sum_{d=1}^{n-1} \left (\sum_{j=1}^n u_ju_{j+d}\right )^2=0
\end{equation}
admits no such solution where each variable $u_j$ assumes $\pm 1$ value.
\end{lemma}
\begin{proof}
This is trivial.
\end{proof}

\medskip

While the above lemma is trivial, it can be combined with the system of equations \eqref{Mtiling}.  Let $S: \hat{G_0}\to \RR$ denote the function defined by the coefficients on the left-hand side of \eqref{S}, i.e.
\begin{equation}\label{Sdef}\sum_{d=1}^{n-1} \left (\sum_{j=1}^n u_ju_{j+d}\right )^2=\sum_\gamma S(\gamma)\vu^\gamma .
\end{equation}

\medskip

Similar to \eqref{Mtiling} we can now write a system of linear equations involving $S$: if $\vu$ generates a cyclic Hadamard matrix then $M(\gamma)=\vu^\gamma$ satisfies the following equations:
\begin{equation}\label{Stiling}
\sum_\rho M(\gamma+\rho)S(\rho)=0  \ \textrm{for \ all} \ \gamma \in \ZZ_2^n.
\end{equation}

\medskip

\begin{lemma}\label{Sequiv}
There exists a $\pm 1$ vector $\vu$ generating a cyclic Hadamard matrix if and only if the homogenous system of linear equations \eqref{Stiling} admits a non-trivial solution $M(\gamma)$.
\end{lemma}

\begin{proof}
If $\vu$ generates a cyclic Hadamard matrix then $M(\gamma)=\vu^\gamma$ satisfies \eqref{Stiling}, yielding a non-trivial solution. In the converse direction, assume $M(\gamma)$ is a non-trivial solution to \eqref{Stiling}. Notice that the left hand side of \eqref{Stiling} is the convolution $S\ast M$ of the functions $S$ and $M$ on the group $\hat \cg$. This means that the convolution  $S\ast M \equiv 0$ on $\hat \cg$. As $M$ is assumed not to be identically zero, taking Fourier transform again we conclude that $\hat S$ must have a zero on $\cg$. This means exactly that there exist a solution $\vu$ to the equation \eqref{S}.
\end{proof}

We can now prove that the linear system of equations \eqref{Mtiling} is an equivalent formulation of the circulant Hadamard conjecture.

\begin{lemma}\label{nolostinf}
Regard each $M(\gamma)$ as a real variable, and consider the system of linear equations determined by \eqref{Mtiling}. The circulant Hadamard conjecture is true for $n$ if and only if this system of equations has full rank, i.e. the only solution is $M(\gamma)=0$ for each $\gamma$.
\end{lemma}

\begin{proof}
One direction is trivial: if $\vu$ generates a circulant Hadamard matrix then $M(\gamma)=\vu^\gamma$ is a non-trivial solution to \eqref{Mtiling}.

\medskip

Conversely, if there exists a non-trivial solution $M(\gamma)$ of \eqref{Mtiling} then $M$ is a fortiori a solution of \eqref{Stiling}, and therefore a circulant  Hadamard matrix exists by Lemma \ref{Sequiv}.
\end{proof}

\medskip

While all the results above are fairly trivial, they do have some {\it philosophical} advantages. First, we can rest assured that Ryser's circulant Hadamard conjecture can be proved or disproved in this manner -- we have not lost any information by setting up the system \eqref{Mtiling}. Second, the circulant Hadamard conjecture is a non-existence conjecture, which can now be transformed to an existence result (i.e. it is enough to exhibit a {\it witness} which proves the non-existence of circulant Hadamard matrices):

\begin{corollary}\label{witness}
The circulant Hadmard conjecture is true for $n$ if and only if there exists real weights $c_{\gamma, d}$ such that
\begin{equation}\label{wit}
\sum_{\gamma, d} c_{\gamma, d}\left (\sum_{j-k= d (mod \ n)} M(\gamma+\pi_j+\pi_k) \right )=M(0)
\end{equation}
\end{corollary}
\begin{proof}
If such weights exist, then \eqref{Mtiling} cannot admit a solution in which $M(0)=1$, and hence there cannot exist a circulant Hadamard matrix of order $n$. Conversely, if such weights do not exist then the linear system \eqref{Mtiling} does not have full rank, so a circulant Hadamard matrix of order $n$ exists by Lemma \ref{nolostinf}.
\end{proof}

\medskip

Therefore we are left with the "simple" task of exhibiting a witness (a set of weights $c_{\gamma, d}$) for each $n$. It is possible to obtain such witnesses by computer for small values of $n$, i.e. $n=8,12, 16, 20, 24$. The problem is that there are always an infinite number of witnesses (a whole affine subspace of them with large dimension), and one should somehow select the "nicest" one, which could be generalized for any $n$.

\medskip

It is natural to exploit the invariance properties of the problem as follows. If $M(\gamma)$ is a non-trivial solution to \eqref{Mtiling} then so is $M_\pi(\gamma)=M(\pi(\gamma))$ where $\pi$ is any cyclic permutation of the coordinates. We can therefore define equivalence classes in $\hat\cg$, regarding $\gamma_1$ and $\gamma_2$ equivalent if they are cyclic permutations of each other. After averaging we can then assume that the required weights $c_{\gamma, d}$ are constant on equivalence classes. Furthermore, if $1\le k \le n-1$ is relatively prime to $n$ then multiplication by $k$ defines an automorphism of the cyclic group $\ZZ_n$. We can regard $\gamma_1$ and $\gamma_2$ equivalent if a coordinate transformation corresponding to multiplication by some $k$ transforms one to the other. Similarly, we can regard $d_1$ and $d_2$ equivalent if GCD($d_1, n$)=GCD($d_2, n$). After averaging again, we can assume that the required witness weights $c_{\gamma, d}$ depend only on the equivalence class of $\gamma$ and that of $d$. However, such restrictions still do not determine the weights $c_{\gamma, d}$ uniquely, and still the witnesses form an affine subspace of large dimension.

\medskip

It is also easy to see that we may restrict our attention without loss of generality to the subgroup $\hat \cg_0=\{\gamma\in \hat \cg : \sum_{j=1}^n\gamma_j\equiv 0 \ (mod \ 2)\}$, because all the terms on the left hand side of \eqref{Mtiling} stay in $\hat \cg_0$ if $\gamma\in \hat \cg_0$. We will call $\sum_{j=1}^n\gamma_j$ the {\it weight} of $\gamma$, and denote it by $|\gamma|$.

\medskip

In the last section of this note we will consider {\it symmetric} polynomials of the variables $u_j$, i.e. expressions of the form
\begin{equation}\label{symw}
\sum_{2|w=0}^n \sum_{|\gamma| = w} d_w M(\gamma).
\end{equation}
That is, only $\gamma\in \hat \cg_0$ are considered in the sum, and the coefficient of $M(\gamma)$ depends on the weight of $\gamma$ only. It is trivial to see that such expressions form a vector space of dimension $\frac{n}{2}+1$, a natural basis of which is given by the single-weight expressions
\begin{equation}\label{sw}
\sum_{|\gamma| = w}  M(\gamma), \ \ \ w=0, 2, 4, \dots n.
\end{equation}
One way to generate an expression of the form \eqref{symw} using the equations \eqref{Mtiling} is the following:
\begin{equation}\label{trivw}
\sum_{|\gamma| = w}  \sum_{d=1}^{n/2}\sum_{j-k= d (mod \ n)}M(\gamma+\pi_j+\pi_k), \ \ \ w=0, 2, 4, \dots n.
\end{equation}
\medskip

\begin{lemma}\label{sdim}
If $4$ divides $n$ then the dimension of the subspace spanned by the expressions \eqref{trivw} in the vector space of the expressions of the form \eqref{symw} is $\frac{n}{2}+1$ if $n\ne 4u^2$, while it is $\frac{n}{2}$ if $n=4u^2$.
\end{lemma}
\begin{proof}
It is easy to see that for any $w$ the left hand side of the expression \eqref{trivw} will contain variables $M(\gamma)$ where the weight $|\gamma|$ is $w-2, w$ or $w+2$. It is therefore easy to express \eqref{trivw} in the basis \eqref{sw} explicitly, as a vector of length $\frac{n}{2}+1$ with only 3 non-zero coordinates. This leads to a tri-diagonal matrix whose rank is $\frac{n}{2}+1$ if $n\ne 4u^2$, while it is $\frac{n}{2}$ if $n=4u^2$. The explicit calculations are left to the reader.
\end{proof}

This lemma leads to the following well-known corollary:

\begin{lemma}
If there exists a cyclic Hadamard matrix of order $n$ then $n$ must be an even square number, $n=4u^2$.
\end{lemma}
\begin{proof}
By Proposition \ref{had} $n$ must be divisible by 4. By Lemma \ref{sdim} we see that the expressions \eqref{trivw} generate the whole space of symmetric polynomials given by \eqref{symw}. In particular, the single variable $M(0)$ is in this subspace, so we conclude that there exists an expansion of the form
\begin{equation}\label{trivwit}
\sum_{|\gamma| = w} c_w \sum_{d=1}^{n/2}\sum_{j-k= d (mod \ n)}M(\gamma+\pi_j+\pi_k)=M(0),
\end{equation}
which is a special case of \eqref{wit}.
\end{proof}

One might object that this is a very difficult way of proving a very easy statement. However, it does have some advantages. First, it rhymes very well with \eqref{wit} and the strategy described in the paragraphs after Lemma \ref{witness}. Namely, put the $\gamma$'s and the $d$'s into some equivalence classes and look for a solution to \eqref{wit} such that the coefficients depend only on the equivalence classes. Second, it "nearly" works even if $n$ is a square: the span of the expressions \eqref{trivw} has dimension $\frac{n}{2}$. One could therefore hope for the following strategy to work. Let us call a linear combination on the left hand side of \eqref{trivwit} "trivial". If we could find a non-trivial linear combination \eqref{wit} such that the result is of the form \eqref{symw}, then it is "very likely" that the dimension of the span would increase to $\frac{n}{2}+1$, which would complete the proof of the general case. It is not at all clear whether such "magic" non-trivial linear combination is easy to find for general $n$, but it is not out of the question.


\begin{thebibliography}{11}

\bibitem{ls1}
\textsc{K. H. Leung, B. Schmidt}:
\newblock{\em The field descent method.}
Des. Codes Cryptogr. {\bf 36}, 171-188 (2005).

\bibitem{ls2}
\textsc{K. H. Leung, B. Schmidt}:
\newblock{\em New restrictions on possible orders of circulant Hadamard matrices.}
Des. Codes Cryptogr.  {\bf 64}, 143-151, (2012).

\bibitem{mubmols}
\textsc{M. Matolcsi, I. Z. Ruzsa, M. Weiner}:
\newblock{\em Systems of mutually unbiased Hadamard matrices containing real and complex matrices.}
Australasian J. Combinatorics, Volume 55 (2013), Pages 35-47.


\bibitem{ryser}
\textsc{H. J. Ryser}:
\newblock{\em Combinatorial Mathematics.}
Wiley, New York (1963).

\bibitem{schmidt1}
\textsc{B. Schmidt}:
\newblock{\em Cyclotomic integers and finite geometries.}
J. Amer. Math. Soc. {\bf 12} (1999) 929-952.


\bibitem{schmidt2}
\textsc{B. Schmidt}:
\newblock{\em Towards Ryser's Conjecture.}
Proceedings of the Third European Congress of Mathematics
(eds C. Casacuberta et al.), Progress in Mathematics 201 (Birkhuser, Boston, 2001) 533-541.

\bibitem{turyn}
\textsc{R. J. Turyn}:
\newblock{\em Character sums and difference sets.}
Pacific J. Math. {\bf 15}, 319-346 (1965).

\end{thebibliography}
\end{document}